\documentclass[11pt,a4paper]{article}

\usepackage[margin=2.5cm]{geometry}

\usepackage[fleqn]{amsmath}
\usepackage{amssymb,latexsym}
\usepackage{amsthm}
\usepackage[section]{placeins}
\usepackage[colorlinks=true,linkcolor=blue,citecolor=blue]{hyperref}
\usepackage{enumerate}

\newtheorem{theorem}{Theorem}[section]
\newtheorem*{thmmatrices}{Theorem \ref{thm:matrices}}
\newtheorem*{thmsmalldiam}{Theorem \ref{thm:smalldiam}}
\newtheorem*{thmdir}{Theorem \ref{thm:dir}}
\newtheorem*{thmgeneral}{Theorem \ref{thm:general}}

\newtheorem{lemma}[theorem]{Lemma}
\newtheorem{proposition}[theorem]{Proposition}

\newenvironment{bullets} {\vspace{-9pt}\begin{itemize}\itemsep0pt} {\end{itemize}\vspace{-9pt}}

\newcommand{\limsupinfty}[1][n]{\limsup\limits_{#1\rightarrow\infty}}
\newcommand{\liminfinfty}[1][n]{\liminf\limits_{#1\rightarrow\infty}}
\newcommand{\Z}{\mathbb{Z}}

\DeclareMathOperator{\Cay}{Cay}
\DeclareMathOperator{\Aut}{Aut}
\title{\textbf{Large Cayley graphs of small diameter}}

\author{
Grahame Erskine\thanks{Open University, Milton Keynes, UK}\footnotemark[1]\\ \texttt{\small grahame.erskine@open.ac.uk}
\and James Tuite\footnotemark[1]\\ \texttt{\small james.tuite@open.ac.uk}
}

\date{}

\begin{document}
\maketitle
\let\thefootnote\relax
\let\thefootnote\relax
\begin{abstract}\noindent
The degree-diameter problem seeks to find the largest possible number of vertices in a graph having given diameter and given maximum degree.
Very often the problem is studied for restricted families of graph such as vertex-transitive or Cayley graphs,
with the goal being to find a family of graphs with good asymptotic properties.
In this paper we restrict attention to Cayley graphs, and study the asymptotics by fixing a small diameter and constructing families of graphs
of large order for all values of the maximum degree. 
Much of the literature in this direction is focused on the diameter two case.
In this paper we consider larger diameters, and use a variety of techniques to derive new best asymptotic constructions for
diameters 3, 4 and 5 as well as an improvement to the general bound for all odd diameters.
Our diameter 3 construction is, as far as we know, the first to employ matrix groups over finite fields in the degree-diameter problem.
\end{abstract}

\section{Introduction}\label{sec:intro}
\subsection{Background and notation}
The goal of the \emph{degree-diameter problem} is to identify the largest possible number $n(d,k)$ of vertices 
in a graph having diameter $k$ and maximum degree $d$.
This is a very general question, and a typical approach is to study restricted versions of the problem, where we consider only graphs of a certain type
such as vertex-transitive or Cayley graphs.
In this paper we consider Cayley graphs, and study both undirected and directed versions of the problem.
For a history and more complete summary of the degree-diameter problem, see the survey paper by Miller and \v{S}ir\'a\v{n}~\cite{miller2005moore}.

Our aim is to study an asymptotic version of the degree-diameter problem. 
We fix a (small) diameter $k$, and ask how large a graph of maximum degree $d$ we can create, then let $d$ go to $\infty$.
To avoid trivial cases, we insist $d\geq 3$ in the case of undirected graphs, and $d\geq 2$ in the directed case.
If our largest graph has order $n(d,k)$ then this value is limited by the well-known Moore bound (see for example~\cite{miller2005moore}) as follows.

\[
n(d,k)\leq M_{d,k}=
\begin{cases}
\displaystyle 1+d\frac{(d-1)^k-1}{d-2} & \qquad\text{(undirected case)}\\[1em]
\displaystyle \frac{d^{k+1}-1}{d-1} & \qquad\text{(directed case)}
\end{cases}
\]

When we restrict attention to a particular family of graphs $\mathcal{G}$, we will denote our limit by $n(d,k;\mathcal{G})$.
Where the context is unambiguous, we will omit the family $\mathcal{G}$ from the notation.

For a given family $\mathcal{G}$ and diameter $k$, we are interested in determining the asymptotics of $n(d,k;\mathcal{G})$ as the degree $d$ tends to infinity.
We note that in both undirected and directed cases, the Moore bound $M_{d,k}$ can be expressed in the form $M_{d,k}=d^k+o(d^k)$ as $d\to\infty$.
To measure the progress of our constructions we define the following quantities.
\begin{align}
  L^-(k;\mathcal{G}) &= \liminfinfty[d] \frac{n(d,k;\mathcal{G})}{d^k} \\
  L^+(k;\mathcal{G}) &= \limsupinfty[d] \frac{n(d,k;\mathcal{G})}{d^k}
\end{align}

As before, we omit the family $\mathcal{G}$ where the context is clear.

We focus on the restricted case of Cayley graphs, so that unless otherwise stated we may take $\mathcal{G}$ to be the family of all such graphs.
We define a \emph{Cayley graph} as follows.

Let $G$ be a finite group with identity element 1 and let $S\subseteq G$ a subset such that $1\notin S$.
Then the \emph{Cayley graph} $\Cay(G,S)$ has the elements of $G$ as its vertex set and each vertex $g$ has an arc to $gs$ for each $s\in S$.
The following properties of $\Cay(G,S)$ are immediate from the definition.
\begin{bullets}
 \item $\Cay(G,S)$ has order $|G|$ and is a regular graph of degree $|S|$.
 \item $\Cay(G,S)$ has diameter at most $k$ if and only if every element of $G$ can be expressed as a product of no more than $k$ elements of $S$.
 \item $\Cay(G,S)$ is an undirected graph if $S=S^{-1}$; otherwise it is a directed graph.
\end{bullets}

\subsection{Existing results}
Much of the existing research tends to focus on the diameter 2 case. For Cayley graphs, the best general result at diameter 2 is due to Abas~\cite{Abas2015a}
yielding $L^-(2)\geq 0.684$. For diameters 3, 4 and 5 the best current results are due to Vetr\'ik~\cite{vetrik345} giving
$L^-(3)\geq 0.187$, $L^-(4)\geq 0.051$ and $L^-(5)\geq 0.024$. 

For larger diameters, no better result is known than that of Macbeth, \v{S}iagiov\'a, \v{S}ir\'a\v{n} and Vetr\'ik~\cite{macbeth2010large}.
This yields $L^-(k)\geq\frac{k}{3^k}$ for any diameter $k$.

\subsection{Outline of new results}
We use two distinct techniques to improve these bounds. In Section~\ref{sec:matgroups} we find an explicit construction of Cayley graphs of 
certain groups of $3\times 3$ matrices over a finite field to give our first result.
\begin{thmmatrices}
In the class of undirected Cayley graphs,
\begin{align*}
L^-(3)\geq \frac{1}{4}&=0.25000
\end{align*}
\end{thmmatrices}

This is as far as we know the first published construction to use matrix groups in this way.

In Section~\ref{sec:semidirect} we generalise the techniques of Bevan, Macbeth, \v{S}iagiov\'a, \v{S}ir\'a\v{n}, Vetr\'ik and others 
by using a construction based on semidirect products. In this way we prove the following.
\begin{thmsmalldiam}
In the class of undirected Cayley graphs,
\begin{align*}
L^-(4)&\geq \frac{60}{5^4} \approx 0.09600 \\
L^-(5)&\geq \frac{60}{4^5} \approx 0.05859 \\
L^-(6)&\geq \frac{78}{4^6} \approx 0.01904 \\
L^-(7)&\geq \frac{168}{4^7} \approx 0.01025 \\
\end{align*}
\end{thmsmalldiam}

In Section~\ref{sec:dir} we extend this method to the case of directed graphs to obtain the following.
\begin{thmdir}
In the class of directed Cayley graphs,
\begin{align*}
L^-(3)&\geq \frac{48}{4^3} = 0.75000 \\
L^-(4)&\geq \frac{36}{3^4} \approx 0.44444 \\
L^-(5)&\geq \frac{120}{3^5} \approx 0.49382 \\
\end{align*}
\end{thmdir}

In Section~\ref{sec:general} we specialise the method to the case where the acting group in the semidirect product is dihedral.
In this way we obtain a new bound for general (odd) diameters as  follows.
\begin{thmgeneral}
In the class of undirected Cayley graphs, for any odd diameter $k$,
\begin{align*}
L^-(k)&\geq \frac{2k}{3^k}
\end{align*}
\end{thmgeneral}

\section{Cayley graphs of matrix groups}\label{sec:matgroups}
Our first construction addresses the diameter 3 case.
Our strategy will be to find a suitable Cayley graph on a group based on a particular subgroup of $SL(3,p)$ for any odd prime $p$.
However, we will be left with some awkward subgroups which our chosen generating set is unable to cover directly.
We therefore begin with two lemmas on diameter 3 Cayley graphs of cyclic and elementary abelian groups.
\begin{lemma}\label{lem:zn}
For any $n\geq 6$ there is a subset $S\subseteq\Z_n$ of cardinality $\displaystyle 6\left\lceil\frac{n^{1/3}}{2}\right\rceil$
such that $\Cay(\Z_n,S)$ has diameter at most 3.
\end{lemma}
\begin{proof}
Let $n\geq 6$ and let $K=\lceil n^{1/3}\rceil$ and $M=\lfloor\frac{K}{2}\rfloor$.
Let $S\subseteq \Z_n$ be the set $\{\pm 1,\pm 2,\ldots,\pm M,$ $\pm K,\pm 2K,\ldots,\pm MK,\pm K^2,\pm 2K^2,\ldots,\pm MK^2\}$.
Then it is easy to see that we can express any element of $\Z_n$ as a sum of at most 3 elements of $S$.
\end{proof}

\begin{lemma}\label{lem:znzn}
For all large $n$, there is a subset $T\subseteq\Z_n\times\Z_n$ of cardinality $9n^{2/3}+o(n^{2/3})$
such that $\Cay(\Z_n\times\Z_n,T)$ has diameter at most 3.
\end{lemma}
\begin{proof}
For the set $T$ we may take the Cartesian product of two copies of the set $S$ from Lemma~\ref{lem:zn}.
\end{proof}

Our diameter 3 construction will be based on finite fields,
which means we can only directly use the construction for degrees which are related to some prime power.
While this construction yields graphs which are valid for an infinite number of degrees and hence can be used to obtain a lower bound on $L^+(k)$,
we would ideally like to extend the validity to all degrees and hence obtain a bound on $L^-(k)$.
Our strategy is to use results from analytic number theory on the distribution of prime numbers to prove that for all sufficiently large degrees $d$,
we can find a prime number such that we can build a graph $\Cay(G,S)$ of degree $d'\leq d$ using our construction.
We then add $d-d'$ generators to our set $S$ yielding a graph of the same order, no larger diameter and degree $d$.
The method hinges on being able to find a prime $p$ such that $d-d'$ is small enough not to affect the asymptotic value of the result.

The method was first used by \v{S}iagiov\'a, \v{S}ir\'a\v{n} and \v{Z}dimalov\'a~\cite{Siran2011large}, and others have since used
similar ideas, for example the first author in~\cite{Erskine2015}.
Because this is such a useful technique we give here a general version of the lemma.

\begin{lemma}\label{lem:prime}
Let $\mathcal{G}$ be a family of groups.
Let $k\geq 2$ and suppose that there exists some $N$ such that for all primes $p\geq N$, 
we can find a group $G(p)\in\mathcal{G}$ and an inverse-closed subset $S(p)\subseteq G(p)$
such that $\Cay(G(p),S(p))$ has diameter $k$. 
Suppose further that there exist positive constants $C,D$ such that as $p\to\infty$, $|G(p)|=Cp^k+o(p^k)$, 
$|S(p)|=Dp+o(p)$ and that for all $p$, $G(p)\setminus S(p)$ contains at least one involution.

Then in the class of Cayley graphs, $\displaystyle L^-(k;\mathcal{G})\geq\frac{C}{D^k}$.
\end{lemma}
\begin{proof}
It suffices to show that for any sufficiently large degree $d$, we can find a Cayley graph of a group in $\mathcal{G}$
with degree $d$, diameter $k$ and order $\frac{C}{D^k}d^k+o(d^k)$.
Let $d$ be a degree large enough so that there exists a prime $p$ such that we can find a group $G(p)$ and a set $S(p)$ satisfying the conditions.
We choose $p$ to be the largest such prime so that $|S(p)|\leq d$.
We now add any inverse-closed set of size $d-|S(p)|$ chosen from $G(p)\setminus S(p)$ to our generating set to obtain a new generating set $S'(p)$.
Note that we can always do this because if we need to add an odd number of generators, we have an involution in $G(p)\setminus S(p)$.

Let $d'=|S(p)|$. Then $d'=Dp+o(p)$.
Now we use the result of Baker, Harman and Pintz~\cite{baker2001difference}
which states that for sufficiently large $x$, we are guaranteed a prime in the interval $(x,x+x^{\theta}]$ where $\theta=0.525$.
This means that $p=\frac{1}{D}d'+o(d')=\frac{1}{D}d+o(d)$.
Then $\Cay(G(p),S'(p))$ has the required properties.
\end{proof}

For any odd prime $p$, we begin with a group $H$ which is the unique non-abelian group of order $p^3$ with exponent $p$.
This has the form $(\Z_p\times\Z_p)\rtimes\Z_p$.
It is well known that the group $H$ can be viewed as the \emph{upper unitriangular} subgroup of $SL(3,p)$, i.e. the subgroup consisting of matrices of the form
$\begin{pmatrix}1&a&b\\0&1&c\\0&0&1\end{pmatrix}$ where $a,b,c$ are arbitrary elements of $GF(p)$.
The group $G$ for our Cayley graph will be a direct product of this group with $\Z_2$.
\begin{lemma}\label{lem:zpzpzp}
Let $p$ be an odd prime. Let $H$ be the upper unitriangular subgroup of $SL(3,p)$ and let $G=H\times\Z_2$. 
Then there is an inverse-closed subset $S$ of $G$ with cardinality $2p+O(p^{2/3})$ such that the Cayley
graph $\Cay(G,S)$ has diameter 3, and $S$ contains neither the identity nor the unique involution of $G$.
\end{lemma}
\begin{proof}
 We construct our generating set $S$ for $G$ as follows. 
For each $x\in GF(p)^*$ we define the following elements of $G$.
\[\alpha_x=\left(\begin{pmatrix}1&x&x\\0&1&0\\0&0&1\end{pmatrix},0\right);\qquad\beta_x=\left(\begin{pmatrix}1&0&x\\0&1&x\\0&0&1\end{pmatrix},1\right)\]
Let $S_1$ be the set consisting of $\alpha_x$ and $\beta_x$ for all non-zero $x\in GF(p)$. Notice that $S_1$ contains neither the identity nor the involution.
We now show that all elements of $G$ of the forms 
$\left(\begin{pmatrix}1&a&b\\0&1&c\\0&0&1\end{pmatrix},0\right),a\neq 0$ and $\left(\begin{pmatrix}1&a&b\\0&1&c\\0&0&1\end{pmatrix},1\right),c\neq 0$
may be expressed as a product of at most 3 elements from $S_1$.

First consider $X=\left(\begin{pmatrix}1&a&b\\0&1&c\\0&0&1\end{pmatrix},0\right),a\neq 0$. There are three cases to consider.
If $b=a+c$ then we choose any $u\notin\{0,c\}$ and then $X=\beta_u\beta_{c-u}\alpha_a$.
Otherwise if $b=a+c+ac$ then again we choose $u\notin\{0,c\}$ and this time $X=\alpha_a\beta_u\beta_{c-u}$.
Otherwise we let $x=c-(b-c)/a+1;y=a;z=(b-c)/a-1$ and then $X=\beta_x\alpha_y\beta_z$.

Now consider $X=\left(\begin{pmatrix}1&a&b\\0&1&c\\0&0&1\end{pmatrix},1\right),c\neq 0$. 
Let $x=(b-a)/c-1;y=c;z=a-(b-a)/c+1$.
If $b=a+c$ then $X=\beta_y\alpha_z$. 
Otherwise if $b=a+c+ac$ then $X=\alpha_x\beta_y$.
Otherwise $X=\alpha_x\beta_y\alpha_z$.

Now we deal with the remaining cases. The elements of the form $\left(\begin{pmatrix}1&0&b\\0&1&c\\0&0&1\end{pmatrix},0\right)$
form a subgroup of $G$ isomorphic to $\Z_p\times\Z_p$.
By Lemma~\ref{lem:znzn} there is a set $S_2$ of size $9p^{2/3}+o(p^{2/3})$ such that each of these can be expressed as a product of at most 3 elements of $S_2$.

Finally, the elements of the form $\left(\begin{pmatrix}1&a&b\\0&1&0\\0&0&1\end{pmatrix},1\right)$ are contained in a subgroup of $G$
isomorphic to $\Z_p\times\Z_p\times\Z_2$.
In a similar way, we can find a set $S_3$ of size $18p^{2/3}+o(p^{2/3})$ such that each of these can be expressed as a product of at most 3 elements of $S_3$.
Letting $S=S_1\cup S_2\cup S_3$ we see that $\Cay(G,S)$ has diameter at most 3 and $|S|=2p+O(p^{2/3})$ as required.
\end{proof}

The main result now follows. 
\begin{theorem}\label{thm:matrices}
In the class of general Cayley graphs, 
\[
L^-(3)\geq\frac{1}{4}
\]
\end{theorem}
\begin{proof}
The graphs in Lemma~\ref{lem:zpzpzp} have order $2p^3$ and degree $2p+O(p^{2/3})$, and satisfy the conditions of Lemma~\ref{lem:prime}.
\end{proof}

\section{A semidirect product construction}\label{sec:semidirect}
\subsection{Motivation}
The results of Section~\ref{sec:matgroups} provide a useful improvement to the asymptotic bound at diameter 3 from Vetr\'ik's $3/16$~\cite{vetrik345} to $1/4$.
For larger diameters, it is possible that other subgroups of matrix groups might yield interesting results.
However, it is unlikely that a general construction covering a range of diameters would be possible with this approach.

To find a more general construction, we are inspired by the ideas of Vetr\'ik~\cite{vetrik345}, Macbeth et al~\cite{macbeth2010large},
Bevan~\cite{BevDeBruijn} and others.
A common strategy of such constructions for a given diameter $k$ is to begin with a $k$-fold direct product of some group $H$, 
and then to permute the coordinate positions in the direct product by means of a semidirect product of $H^k$ by some other group $K$.

We recall first the definition of a semidirect product, choosing here a notation convenient for our needs.
Given two groups $G$ and $K$ and a group homomorphism $\varphi:K\to\Aut(G)$, the semidirect product $G\rtimes_{\varphi}K$
is the group with element set the Cartesian product $G\times K$ and multiplication defined by:
\[
(g_1,k_1)(g_2,k_2)=(g_1^{\varphi(k_2)}g_2,k_1 k_2)
\]
where the superscript on $g_1$ indicates the image of $g_1$ under the automorphism $\varphi(k_2)$ of $G$.

In a $k$-fold direct product $H^k$, any permutation of the $k$ coordinate positions is an automorphism of the group.
These automorphisms of $H^k$ form a subgroup of its full automorphism group, this subgroup being isomorphic to the symmetric group $S_k$.
We restrict ourselves in our semidirect products $H^k\rtimes_{\varphi}K$ to homomorphisms $\varphi$ into this restricted subgroup.

Our goal is again to find a lower bound on the quantity $L^-(k)$, as defined in Section~\ref{sec:intro}, for certain fixed values of $k$.
To achieve this we first fix a diameter $k$, and then try to construct an infinite sequence of Cayley graphs of degree $sm$ and asymptotic order $m^k n$ 
for every $m\geq 2$ and for some fixed constants $n,s$.

We begin the discussion with an example at diameter 6 which should help clarify the overall method.

\subsection{Diameter 6 example}\label{sec:diam6}
Let $H$ be an abelian group of order $m$. For the purposes of our construction we take $H=\Z_m$ and use additive notation for the group operation.
Let $k=6$ and denote the 6-fold direct product of $H$ by $H^6$.
Let $\rho$ be the automorphism of $H^6$ which maps the element $(x_1,x_2,x_3,x_4,x_5,x_6)$ to $(x_6,x_1,x_2,x_3,x_4,x_5)$.
Let $K$ be the group $\Z_{36}$ and let $\varphi:K\to \Aut(H^k)$ be the group homomorphism given by $\varphi(r)=\rho^r$.
Let $G=H^6\rtimes_{\varphi}K$.

We write the elements of $G$ in the form $(x_1,x_2,x_3,x_4,x_5,x_6;y)$ where each $x_i\in H$ and $y\in K$.
We construct our generating set as follows. For each $x\in H$ we define:
\begin{align*}
a(x)&=(0,0,0,0,0,x;1)\\
A(x)&=(0,0,0,0,x,0;-1)\\
b(x)&=(0,x,x,0,0,x;4)\\
B(x)&=(0,x,0,x,x,0;-4)
\end{align*}

Then the generating set is:
\[
X=\bigcup_{x\in H}\{a(x),A(x),b(x),B(x)\}
\]

Note that since $a(x)^{-1}=A(-x)$ and $b(x)^{-1}=B(-x)$, the set $S$ is inverse-closed.

We claim that the graph $\Cay(G,X)$ has diameter 6. 
To substantiate this, it suffices to show that every element of $G$ can be expressed as a product of at most 6 elements of $X$.
For a given $y\in K$, we can express the element $(x_1,x_2,x_3,x_4,x_5,x_6;y)\in G$ via the products in Figure~\ref{fig:soldiam6}.
For $y=18\ldots 35$, we may obtain expressions simply by inverting the appropriate products.
Thus $\Cay(G,X)$ has diameter 6 as claimed.

\begin{figure}\small
\begin{align*}
y=0:\  & a(-x_2+x_3+x_5)a(-x_2+x_3+x_4)b(x_3)\\[-0.5em]
 & A(-x_3+x_6)A(x_1)B(x_2-x_3)\\
y=1:\  & b(x_3)b(x_2-x_4)A(-x_2-x_3+x_5)\\[-0.5em]
 & A(-x_3+x_6)A(x_1-x_2+x_4)B(x_4)\\
y=2:\  & b(x_6)a(x_3-x_4+x_6)a(x_2)\\[-0.5em]
 & a(x_1-x_4)B(x_4-x_6)A(x_5)\\
y=3:\  & a(x_2)a(x_1)b(x_6)\\[-0.5em]
 & A(x_3-x_6)A(x_4)A(x_5-x_6)\\
y=4:\  & a(-x_1+x_3)a(x_1+x_2-x_4)B(-x_1+x_4)\\[-0.5em]
 & a(x_1-x_4+x_5)b(x_1)a(x_6)\\
y=5:\  & a(x_2+x_4+x_5-x_6)a(-x_2+x_3-x_5)B(x_5)\\[-0.5em]
 & A(x_1+x_2+x_5-x_6)b(-x_2-x_5+x_6)b(x_2)\\
y=6:\  & a(x_5)a(x_4)a(x_3)\\[-0.5em]
 & a(x_2)a(x_1)a(x_6)\\
y=7:\  & a(x_1-x_3+x_6)a(2x_1-x_3-x_4+x_5)b(x_1)\\[-0.5em]
 & b(-x_1+x_3)a(x_1+x_2-x_4)B(-x_1+x_4)\\
y=8:\  & a(x_1+x_3-x_6)b(x_3)a(x_2)\\[-0.5em]
 & b(-x_3+x_6)A(x_3+x_4-x_6)A(-x_3+x_5)\\
y=9:\  & a(x_2-x_3)a(x_1)a(-x_3+x_6)\\[-0.5em]
 & a(x_5)a(x_4)b(x_3)\\
y=10:\  & a(x_3)b(x_2-x_4)A(-x_2+x_5)\\[-0.5em]
 & b(x_4)a(x_1-x_2+x_4)a(x_6)\\
y=11:\  & a(x_4)b(-x_1+x_2)A(-x_3+x_6)\\[-0.5em]
 & b(x_1-x_2+x_3)b(x_1)A(-2x_1+x_2-x_3+x_5)\\
y=12:\  & a(x_5)a(-x_1+x_4)a(x_1+x_3-x_6)\\[-0.5em]
 & a(x_1+x_2-x_6)b(x_1)b(-x_1+x_6)\\
y=13:\  & a(-2x_2+x_3-x_4+x_6)a(x_5)b(-x_2+x_3)\\[-0.5em]
 & A(x_1-x_4)b(x_2-x_3+x_4)b(x_2)\\
y=14:\  & b(-x_3+x_6)b(x_3)b(x_3+x_4-x_6)\\[-0.5em]
 & A(-x_1+x_2-3x_3-x_4+2x_6)b(x_1+x_3-x_6)A(-x_1-2x_3-x_4+x_5+2x_6)\\
y=15:\  & a(x_2-x_3-x_5)a(x_1-x_4+x_5)a(-x_3-x_4+x_5+x_6)\\[-0.5em]
 & b(x_5)b(x_4-x_5)b(x_3)\\
y=16:\  & a(-2x_2+x_3-x_4+2x_5+x_6)b(x_5)b(x_2+x_4-x_5-x_6)\\[-0.5em]
 & A(x_1-x_4-x_5)b(-x_2+x_5+x_6)b(x_2-x_5)\\
y=17:\  & B(x_1-2x_2-x_3+2x_5-x_6)B(-x_2+x_5)B(x_2+x_3-x_5)\\[-0.5em]
 & a(x_1-5x_2-2x_3+x_4+4x_5-2x_6)B(x_2-x_5+x_6)B(-x_1+3x_2+x_3-2x_5+x_6)
\end{align*}
\caption{Solution for diameter 6}
\label{fig:soldiam6}
\end{figure}

To illustrate the multiplication rules we show here an example from the solution. In the case of $y=3$ we clam that
\[(x_1,x_2,x_3,x_4,x_5,x_6;3)=a(x_2)a(x_1)b(x_6)A(x_3-x_6)A(x_4)A(x_5-x_6).\]
Expanding the right hand side one step at a time we get the following.
\begin{align*}
&a(x_2)a(x_1)\\
&\quad=(0,0,0,0,0,x_2;1)(0,0,0,0,0,x_1;1)
\end{align*}
The multiplication rule is that we rotate the first 6 coordinates of the first term by the final coordinate of the second term, then add. So we get:
\begin{align*}
&a(x_2)a(x_1)\\
&\quad=(x_2,0,0,0,0,x_1;2)
\end{align*}
We continue in this way.
\begin{align*}
&a(x_2)a(x_1)b(x_6)\\
&\quad=(x_2,0,0,0,0,x_1;2)(0,x_6,x_6,0,0,x_6;4)\\
&\quad=(0,x_6,x_6,x_1,x_2,x_6;6)
\end{align*}
\begin{align*}
&a(x_2)a(x_1)b(x_6)A(x_3-x_6)\\
&\quad=(0,x_6,x_6,x_1,x_2,x_6;6)(0,0,0,0,x_3-x_6,0;-1)\\
&\quad=(x_6,x_6,x_1,x_2,x_3,0;5)
\end{align*}
\begin{align*}
&a(x_2)a(x_1)b(x_6)A(x_3-x_6)A(x_4)\\
&\quad=(x_6,x_6,x_1,x_2,x_3,0;5)(0,0,0,0,x_4,0;-1)\\
&\quad=(x_6,x_1,x_2,x_3,x_4,x_6;4)
\end{align*}
\begin{align*}
&a(x_2)a(x_1)b(x_6)A(x_3-x_6)A(x_4)A(x_5-x_6)\\
&\quad=(x_6,x_1,x_2,x_3,x_4,x_6;4)(0,0,0,0,x_5-x_6,0;-1)\\
&\quad=(x_1,x_2,x_3,x_4,x_5,x_6;3)
\end{align*}

Since $|G|=36m^6$ and $|X|=4m$, it follows that for every degree $d$ of the form $4m$, there exists a Cayley graph of diameter 6
and order $36d^6/4^6$.
To cover the cases $d\equiv 1,2,3\pmod 4$ we may simply add one more involution from $G$ and/or one more pair of mutually inverse elements to our set $X$.
We have therefore proved the following result.
\begin{proposition}
In the class of Cayley graphs,
\[
L^-(6)\geq \frac{36}{4^6}\approx 0.00878
\]
\end{proposition}
This is, as far as we know, the first specific result for diameter 6 and is an improvement on the bound of $6/3^6\approx 0.00823$ from \cite{macbeth2010large}.
However, the method is capable of generalisation and we now describe the full construction.

\subsection{The general construction}
We begin by drawing out the key features of the construction in Section~\ref{sec:diam6}.
Recall that $H=\Z_m$, $K=\Z_{36}$ and $k=6$. We define $n=|K|$, so $n=36$ in our example.
Finally, $\varphi:K\to \Aut(H^k)$ is the group homomorphism given by $\varphi(r)=\rho^r$ and $G=H^k\rtimes_{\varphi}K$.

Our generating set was constructed as follows. We have a set $S=\{s_1,s_2,s_3,s_4\}=\{1,-1,4,-4\}$ which is a subset of $K$ of cardinality 4. 
It can readily be checked that the set $S$ has the property that every element of $K$ can be expressed as a sum of exactly $k$ elements of $S$. 
Moreover, the sums satisfy the further restriction that no element of $S$ appears consecutively with its inverse.
For example, from the table above for the case $y=3$ we have $3=1+1+4-1-1-1$.

We have a set $V=\{v_1,v_2,v_3,v_4\}=\{000001,000010,011001,010110\}$ of 4 non-zero 0/1 vectors of length $k=6$. 
This set has two important properties. The first is that $v_2=v_1^{\rho^{-s_1}}$ and $v_4=v_1^{\rho^{-s_3}}$, 
where as before $\rho$ represents a right rotation of one place in the coordinates. 
This ensures that our generating set defined below will be inverse-closed.
The second is that the vectors have been carefully chosen to ensure that our eventual graph will have diameter 6.

For every $x\in H$, we define $v_i(x)$ to be the element of $H^k$ with $x$ in every coordinate position where $v_i$ has a 1, and 0 otherwise. 
We now define our generating set $X$ to consist of four sets of elements of $G$ as follows.
\begin{align*}
a(x)=(v_1(x);s_1) \quad&\text{for all }x\in H\\
A(x)=(v_2(x);s_2) \quad&\text{for all }x\in H\\
b(x)=(v_3(x);s_3) \quad&\text{for all }x\in H\\
B(x)=(v_4(x);s_4) \quad&\text{for all }x\in H\\
\end{align*}

\vspace*{-4ex}
Note that because of the forms of the vectors $v_i$ explained above, we have $a(x)^{-1}=A(-x)$, $b(x)^{-1}=B(-x)$ and so $X$ is an inverse-closed subset of $G$.

The most awkward part of the process is to determine how to express any possible element of our group $G$ as a product of $k$ of our generators.
To determine how this can be done we proceed as follows.
Let $\mathbf{x}=(x_1,x_2,\ldots,x_k)$ be an arbitrary element of $H^k$.
We must show, for each $i=0\ldots n-1$, that we can express any element of $G$ of the form $(\mathbf{x};i)$ as a product of $k$ generators.
Since the generator set is inverse-closed we need only check $i\leq\lfloor\frac{n}{2}\rfloor$.
To find products which work we proceed as follows for each such $i$.
\begin{enumerate}[a]
	\item Find all possible ways in which $i$ can be expressed as a sum of $k$ elements of $S$ (ignoring order in the sum).
	\item Find all unique ways to order the elements in this sum, say $T=(t_1,t_2,\ldots,t_k)$ with each $t_j\in S$ and $\sum t_j=i$.
	We insist also that $t_{j+1}\neq -t_j$ for $j=1\ldots k-1$.
	\item For each $T$, find the vector $U=(u_1,u_2,\ldots,u_k)$ of $k$ numbers chosen from $\{1,2,3,4\}$ such that $t_j=s_{u_j}$ for each $j$.
	That is to say, we identify in order those elements of $S$ involved in the sum. At this point we know our product must have the form
	$(v_{u_1}(y_1);s_{u_1})(v_{u_2}(y_2);s_{u_2})\cdots(v_{u_k}(y_k);s_{u_k})$ for some $\mathbf{y}=(y_1,y_2,\ldots,y_k)$.
	\item To determine whether there is a solution, we compute the \emph{mapping matrix} $M$ such that $\mathbf{y}M=\mathbf{x}$.
	If $M$ is invertible over $\Z$ (i.e. it has determinant $\pm 1$) we have found a solution for $i$, otherwise we proceed with the search.
\end{enumerate}
In the final step, it is easy to see that the mapping matrix $M$ has the following form:
\[M=\begin{pmatrix}v_{u_1}^{\rho^{r_1}}\\[1em]v_{u_2}^{\rho^{r_2}}\\[1em]\vdots\\[1em]v_{u_k}^{\rho^{r_k}} \end{pmatrix}^T;\qquad r_w=\sum_{j<w}t_j\]

The elements of this construction which can be generalised are as follows.
\begin{enumerate}[(i)]
	\item The target diameter of our Cayley graph could be any $k>2$.
	\item The group $K$ could be an arbitrary group of order $n$ rather than being restricted to cyclic groups.
	\item The size $|S|$ of the set of elements of $K$ need not be 4.
	\item The homomorphism $\varphi$ in the semidirect product could be any non-trivial homomorphism from our group $K$ to the group of coordinate
	permutations of $H^k$.
	\item Our set $V$ of 0/1 vectors could be any set, provided the resulting set of generators is inverse-closed.
\end{enumerate}

It is clear that with the large number of variables, and the relatively complex nature of the construction,
some form of automated search for feasible solutions is essential.
We outline the search algorithm below. We begin with the following inputs:
\begin{itemize}
	\item A target diameter $k$.
	\item A set size $s=|S|$.
	\item A target order $n$ for our group $K$.
\end{itemize}

Given these parameters, we run the search using a GAP~\cite{GAP4} script as follows.
\begin{enumerate}[1.]
	\item Find all groups $K$ of order $n$ from the small groups library.
	\item For each $K$, find (up to conjugacy) all possible homomorphisms $\varphi$ from $K$ to $S_k$. 
	(To avoid trivial cases, we consider only homomorphisms whose image has no fixed point.)
	\item For each $K$, find all sets $S$ of size $s$ with the property that any element of $K$ can be written as a product of exactly $k$ elements of $S$.
	\item For each combination of $\varphi$ and $S$, find all possible sets $V=\{v_1,\ldots,v_s\}$ of 0/1 vectors of length $k$,
	such that, given the elements of $S$, the resulting generating set will be inverse-closed.
\end{enumerate}

For each viable combination found, we then search for a solution using a modified version of the diameter 6 example.
So for each element $i\in K$ we test whether the following procedure succeeds.
\begin{enumerate}[(a)]
	\item Find all ways to express $i$ as a product of $k$ elements of $S$, say $T=(t_1,t_2,\ldots,t_k)$ with each $t_j\in S$ and $\sum t_j=i$.
	As before, we insist that $t_{j+1}\neq t_j^{-1}$ for $j=1\ldots k-1$.
	\item For each $T$, compute the vector $U$ of $k$ numbers chosen from $1\ldots s$ such that $t_j=s_{u_j}$ for each $j$.
	So we know our product must have the form
	$(v_{u_1}(y_1);s_{u_1})(v_{u_2}(y_2);s_{u_2})\cdots(v_{u_k}(y_k);s_{u_k})$ for some $\mathbf{y}=(y_1,y_2,\ldots,y_k)$.
	\item To determine whether there is a solution we again compute the mapping matrix $M$ such that $\mathbf{y}M=\mathbf{x}$.
	If $M$ is invertible over $\Z$ we have found a solution for $i$, otherwise we proceed with the search.
\end{enumerate}
If this procedure finds a solution for all $i\in K$, then our search has yielded a positive result.

When a solution has been found, we know that for any $m$, we can create a Cayley graph of diameter $k$, order $m^k n$ and degree $sm$.
Thus by the same argument as in the diameter 6 example, we will have proved that in the class of Cayley graphs:
\[
L^-(k)\geq\frac{n}{s^k}
\]
The object now is to choose the parameters for the search in such a way that we can improve the existing asymptotic bounds.
The following sections describe our best results.

\subsection{Undirected graphs}
\subsubsection{Diameters 2 and 3}
For diameter 2, our method will never produce a useful result. This is because our construction requires us to be able to express every element
of our group $K$ as a product of $k$ elements chosen from $S$ so that no element follows its inverse in the product.
With $k=2$ this is clearly impossible.

For diameter 3, our best results for set sizes $s=4,5,6,7$ are shown in Table~\ref{tab:sddiam3}. (There were no useful solutions with $s=3$.)
The best existing published result is by Vetr\'ik~\cite{vetrik345} giving $L^-(3)\geq\frac{3}{16}$.
Although our results improve on that, we are unable to do better than the specific diameter 3 construction from Section~\ref{sec:matgroups} above.
\begin{table}\centering
\begin{tabular}{|cccc|}
\hline
Set size $s$ & Group order $n$ & Group $K$ & $L^-(3)$ bound\\
\hline
4 & 12 & $\Z_{12}$ & $12/4^3\approx 0.18750$\\
5 & 24 & $S_4$ & $24/5^3\approx 0.19200$\\
6 & 48 & $(\Z_4\times\Z_4)\rtimes\Z_3$ & $48/6^3\approx 0.22222$\\
7 & 72 & $(\Z_2^2\rtimes\Z_9)\rtimes\Z_2$ & $72/7^3\approx 0.20991$\\
\hline
\end{tabular}
\caption{Best results for undirected graphs of diameter 3}
\label{tab:sddiam3}
\end{table}

\subsubsection{Diameter 4}
For diameter 4, the increasing size of the search space means that we were only able to search for solutions with set sizes of 3, 4 and 5.
The results are summarised in Table~\ref{tab:sddiam4}.
The best existing published result is again by Vetr\'ik~\cite{vetrik345} giving $L^-(4)\geq\frac{32}{5^4}\approx 0.05120$.
For set sizes 4 and 5, we obtain results better than that bound.

\begin{table}\centering
\begin{tabular}{|cccc|}
\hline
Set size $s$ & Group order $n$ & Group $K$ & $L^-(4)$ bound\\
\hline
3 & 4 & $\Z_4$ & $4/3^4\approx 0.04938$\\
4 & 24 & $S_4$ & $24/4^4\approx 0.09375$\\
5 & 60 & $\Z_{15}\rtimes\Z_4$ & $60/5^4\approx 0.09600$\\
\hline
\end{tabular}
\caption{Best results for undirected graphs of diameter 4}
\label{tab:sddiam4}
\end{table}

We note that in contrast to the diameter 6 example construction above which used a cyclic group $K$, the groups found by the computer search are
not at all obvious and the combination of group $K$, homomorphism $\varphi$, set $S$ and vectors $V$ lead to a solution which is complex and lengthy to tabulate.
For reasons of brevity therefore we omit all the full solutions here; the interested reader will find complete tabulations of the optimal solutions for diameters 3--7 in an ancillary file at~\cite{ErskineTuite2016}.

The simplest solution to describe, although not the one yielding the largest value, uses a set of size 4. 
In this case we are fortunate that the group is $S_4$ and the homomorphism $\varphi$ is simply the identity mapping.
We therefore illustrate the results by tabulating this solution below in the same format as our diameter 6 example in Figure~\ref{fig:soldiam6}.

\begin{figure}
\begin{align*}
K&=S_4\\
S&=\{(2\ 3\ 4), (2\ 4\ 3), (3\ 4), (1\ 2)\}\\
V&=\{1010,1100,0100,1110\}\\
a(x)&=(x,0,x,0;(2\ 3\ 4))\\
A(x)&=(x,x,0,0;(2\ 4\ 3))\\
b(x)&=(0,x,0,0;(3\ 4))\\
c(x)&=(x,x,x,0;(1\ 2))\\
y=():\quad & b(x_2-x_3-x_4)c(x_4)b(x_1-x_3-x_4)c(x_3)\\
y=(1\ 2):\quad & a(x_2-x_3-x_4)a(x_4)a(-x_1+x_2-x_4)c(x_1-x_2+x_3+x_4)\\
y=(1\ 3):\quad & b(x_2-x_3)c(x_4)a(-x_1+x_3-x_4)c(x_1)\\
y=(1\ 4):\quad & b(-x_1+x_2)c(x_4)A(x_1-x_3-x_4)c(x_3)\\
y=(2\ 3):\quad & A(x_2-x_3-x_4)c(x_4)b(x_1-x_2)c(x_3)\\
y=(2\ 4):\quad & a(x_4)a(x_2)b(-x_1+x_2+x_3+x_4)a(x_1-x_2-x_4)\\
y=(3\ 4):\quad & b(-x_1+x_2+x_3+x_4)a(x_1-x_3-x_4)a(x_4)a(x_3)\\
y=(1\ 2)(3\ 4):\quad & A(-x_1+x_2)b(x_1-x_2+x_4)A(x_1-x_3)c(x_3)\\
y=(1\ 3)(2\ 4):\quad & a(-x_2+x_3)c(x_2-x_3+x_4)a(-x_1+x_3-x_4)c(x_1)\\
y=(1\ 4)(2\ 3):\quad & b(x_3)A(x_1-x_2)c(-x_1+x_2+x_4)A(x_1-x_4)\\
y=(1\ 2\ 3):\quad & a(-x_1+x_2)b(x_4)A(x_2-x_3)c(x_1-x_2+x_3)\\
y=(1\ 3\ 2):\quad & b(x_1-x_2-x_4)c(x_3)A(-x_3+x_4)A(x_2)\\
y=(1\ 2\ 4):\quad & b(-x_2+x_3+x_4)a(x_2-x_3)a(-x_1+x_3)c(x_1)\\
y=(1\ 4\ 2):\quad & b(x_1-x_3-x_4)c(x_2)a(-x_2+x_4)a(x_3)\\
y=(1\ 3\ 4):\quad & a(x_2)b(x_1+x_2-x_3)c(x_4)a(-x_2+x_3-x_4)\\
y=(1\ 4\ 3):\quad & A(x_1-x_2-x_3)c(x_3)b(-x_1+x_2+x_4)A(x_2)\\
y=(2\ 3\ 4):\quad & b(-x_1+x_2+x_3)a(x_1-x_2)b(x_4)A(x_2)\\
y=(2\ 4\ 3):\quad & b(x_4)A(x_3)A(x_1-x_3)b(-x_1+x_2+x_3)\\
y=(1\ 2\ 3\ 4):\quad & b(x_3)A(x_1-x_4)c(x_4)b(-x_1+x_2)\\
y=(1\ 4\ 3\ 2):\quad & b(x_1-x_2-x_3)c(x_3)b(-x_3+x_4)A(x_2)\\
y=(1\ 2\ 4\ 3):\quad & b(-x_2+x_3+x_4)a(x_2-x_3)b(x_1-x_3)c(x_3)\\
y=(1\ 3\ 4\ 2):\quad & c(x_4)b(x_3-x_4)A(x_1-x_4)b(-x_1+x_2+x_4)\\
y=(1\ 3\ 2\ 4):\quad & a(-x_1+x_2+x_4)c(x_1-x_2+x_3-x_4)A(-x_3+x_4)A(x_2)\\
y=(1\ 4\ 2\ 3):\quad & a(x_3)a(-x_1+x_2)c(x_1-x_2-x_3+x_4)A(x_2+x_3-x_4)
\end{align*}
\caption{Solution for diameter 4}
\label{fig:soldiam4}
\end{figure}

\subsubsection{Diameter 5}
For diameter 5 we were able to search for solutions with set sizes of 3 and 4, with the results summarised in Table~\ref{tab:sddiam5}.
As before, the best existing published result is by Vetr\'ik~\cite{vetrik345} giving $L^-(5)\geq\frac{25}{4^5}\approx 0.02441$.
For set size 3, we have a marginal improvement and at set size 4 our best solution is a substantial increase.

\begin{table}\centering
\begin{tabular}{|cccc|}
\hline
Set size $s$ & Group order $n$ & Group $K$ & $L^-(5)$ bound\\
\hline
3 & 6 & $S_3$ & $6/3^5\approx 0.02469$\\
4 & 60 & $A_5$ & $60/4^5\approx 0.05859$\\
\hline
\end{tabular}
\caption{Best results for undirected graphs of diameter 5}
\label{tab:sddiam5}
\end{table}

\subsubsection{Diameters 6 and 7}
For diameters 6 and 7 we were again able to search for solutions with set sizes of 3 and 4, 
with the results summarised in Tables~\ref{tab:sddiam6} and~\ref{tab:sddiam7}.
There are no specific published results at diameters 6 and 7.
The best available published result comes from the general construction of Macbeth, \v{S}iagiov\'{a}, \v{S}ir\'{a}\v{n} and Vetr\'{i}k~\cite{macbeth2010large}
which yields $L^-(6)\geq\frac{6}{3^6}\approx 0.00823$ and $L^-(7)\geq\frac{7}{3^7}\approx 0.00320$.

Recall that our diameter 6 example with a set size of 4 already yielded an improvement to $\approx 0.00878$,
but with the aid of the computer search we are able to more than double this figure.
At diameter 7 our best result is now more than three times that obtained by the published general construction.
\begin{table}\centering
\begin{tabular}{|cccc|}
\hline
Set size $s$ & Group order $n$ & Group $K$ & $L^-(6)$ bound\\
\hline
3 & 12 & $A_4$ & $12/3^6\approx 0.01646$\\
4 & 78 & $\Z_2\times(\Z_{13}\rtimes\Z_3)$ & $78/4^6\approx 0.01904$\\
\hline
\end{tabular}
\caption{Best results for undirected graphs of diameter 6}
\label{tab:sddiam6}
\end{table}

\begin{table}\centering
\begin{tabular}{|cccc|}
\hline
Set size $s$ & Group order $n$ & Group $K$ & $L^-(7)$ bound\\
\hline
3 & 14 & $D_{14}$ & $14/3^7\approx 0.00640$\\
4 & 168 & $\Z_8\times(\Z_7\rtimes\Z_3)$ & $168/4^7\approx 0.01025$\\
\hline
\end{tabular}
\caption{Best results for undirected graphs of diameter 7}
\label{tab:sddiam7}
\end{table}

\subsubsection{Summary}
We collect the results above into a single theorem.
\begin{theorem}\label{thm:smalldiam}
In the class of undirected Cayley graphs,
\begin{align*}
L^-(4)&\geq\frac{60}{5^4}\approx 0.09600\\
L^-(5)&\geq\frac{60}{4^5}\approx 0.05859\\
L^-(6)&\geq\frac{78}{4^6}\approx 0.01904\\
L^-(7)&\geq\frac{168}{4^7}\approx 0.01025\\
\end{align*}
\end{theorem}
\subsection{Directed graphs}\label{sec:dir}
In the directed case, the best currently available results are by Vetr\'ik~\cite{Vetrik2012} who shows that $L^-(2)\geq 8/9$ and for $k\geq 3$, $L^-(k)\geq k/2^k$.
The search method we used for undirected Cayley graphs can be modified to search for Cayley digraphs.
The only substantial difference is that our generating set $X$ need not be inverse-closed.
This has two major consequences for the search:
\begin{itemize}
	\item The set $S$ of elements of $K$ need not be inverse-closed.
	\item The set $V$ of 0/1-vectors is not restricted by the requirement that the resulting generating set be inverse-closed.
\end{itemize}

These consequences taken together result in a substantial increase in the search space for a given set of parameters.
Due to this effect, we were only able to search a limited range of set sizes (2, 3 and 4) for diameters 3, 4 and 5. As in the undirected case, full tabulations of the optimal solutions may be found in the ancillary file at~\cite{ErskineTuite2016}.
The best results are summarised here in Table~\ref{tab:sdddiam} and in the following theorem.

\begin{theorem}\label{thm:dir}
In the class of directed Cayley graphs,
\begin{align*}
L^-(3)&\geq \frac{48}{4^3} = 0.75000 \\
L^-(4)&\geq \frac{36}{3^4} \approx 0.44444 \\
L^-(5)&\geq \frac{120}{3^5} \approx 0.49382 \\
\end{align*}
\end{theorem}

Our new constructions are able to better the directed graph bounds of Vetr\'ik~\cite{Vetrik2012} at diameters 3, 4 and 5.

In general as one would expect, removing the restriction on generating sets results in bounds which are much better than the corresponding undirected bounds.

The current table has the curious feature that the best result we were able to find for diameter 5 is better than that for diameter 4.
This is counter-intuitive, but may simply be a consequence of the restricted space that we were able to search.

\begin{table}\centering
\begin{tabular}{|ccccc|}
\hline
Diameter $k$ & Set size $s$ & Group order $n$ & Group $K$ & $L^-(k)$ bound\\
\hline
3 & 4 & 48 & $\Z_2\times S_4$ & $48/4^3\approx 0.75000$\\
4 & 3 & 36 & $\Z_3\times A_4$ & $36/3^4\approx 0.44444$\\
5 & 3 & 120 & $S_5$ & $120/3^5\approx 0.49382$\\
\hline
\end{tabular}
\caption{Best results for directed graphs}
\label{tab:sdddiam}
\end{table}

\section{Large Cayley graphs for odd diameters \texorpdfstring{$k \geq 7$}{}}\label{sec:general}

The results of the previous section may be thought of as providing an outline method for the construction of large Cayley graphs of fixed diameter,
which we use to improve the known asymptotic bounds for certain fixed values of the diameter $k$.
The best results for a given diameter appear to arise from a somewhat unpredictable combination of the parameters in the construction.
In this section we change direction slightly and use the same outline method to develop a construction which is valid for any arbitrary odd diameter $k\geq 7$.

The best general construction for undirected graphs in the literature to date is by Macbeth, \v{S}iagiov\'a, \v{S}ir\'a\v{n} and Vetr\'ik~\cite{macbeth2010large},
which yields $L^-(k)\geq\frac{k}{3^k}$ for any diameter $k \geq 3$. 
Our construction uses the method of Section~\ref{sec:semidirect} by fixing the right hand group $K$ to be a dihedral group.
By finding suitable parameters for the generating set, we will show that this bound can be improved, for \emph{odd} diameters $k \geq 7$, 
to $L^-(k)\geq\frac{2k}{3^k}$.

The remainder of this section outlines the construction which leads to this result. 
We note that in the \emph{directed} case, for odd diameters, 
a recent result of Abas and Vetr\'\i k~\cite{Abas2017} yields $L^-(k)\geq\frac{k}{2^{k-1}}$ for odd $k\geq 3$.
In the description of our method we follow some of the notation and terminology from that paper.

Let $k=2q+1$, where $q \geq 3$. Let $H$ be the cyclic group $\Z_m$ of order $m \geq 2$ written additively.
We will label the coordinates of the $k$-fold direct product $H^k$ by the elements $0,1,\ldots,k-1$ of $\Z_k$.
Let $D_{2k} = \langle r,s\,|\,r^k=s^2=1,srs=r^{-1}\rangle$ be the dihedral group of order $2k$. 
Every element of $D_{2k}$ may be written uniquely in the form $r^is^j$ with $0\leq i\leq k-1$ and $j\in\{0,1\}$. Define automorphisms $\rho$ and $\sigma$ of $H^k$ by
\begin{align*}
(x_0,x_1,x_2,\ldots,x_{k-1})^{\rho} &= (x_{k-1},x_0,x_1,\ldots,x_{k-2})\\
(x_0,x_1,x_2,\ldots,x_{k-1})^{\sigma} &= (x_{k-1},x_{k-2},\ldots,x_1,x_0)
\end{align*}
for all $\mathbf{x} \in H^k$.  There is a homomorphism $\varphi : D_{2k} \to \Aut(H^k)$ given by $\varphi (r^is^j) = \rho ^i \sigma ^j$.  If we let $\psi _r, \psi _{r^{-1}}, \psi _s$ be the permutations of the coordinate positions of $\Z _k$ corresponding to $\varphi (r), \varphi (r^{-1})$ and $\varphi (s)$, then $\psi _r(x) = x+1, \psi _{r^{-1}}(x) = x - 1, \psi _s(x) = k-1-x$, where addition and subtraction are taken modulo $k$.  Then $\psi _r \psi _s \psi _{r^{-1}} \psi _s = \psi _r^2$.

We will proceed to construct a Cayley graph on the group $G = H^k \rtimes _{\varphi} D_{2k}$. 
We define the semidirect product as in Section~\ref{sec:semidirect} so that $G$ has elements $(\mathbf{x};y), \mathbf{x} \in H^k, y \in D_{2k}$
and the multiplication rule of $G$ is
\[(\mathbf{x}_1;y_1)(\mathbf{x}_2;y_2) = (\mathbf{x}_1^{\varphi(y_2)}+\mathbf{x}_2;y_1y_2)\]
for $\mathbf{x}_1, \mathbf{x}_2 \in H^k, y_1,y_2 \in D_{2k}$.

Let $S = \{r,r^{-1},s\}$.   We now define a corresponding family $V=\{\mathbf{v}_r,\mathbf{v}_{r^{-1}},\mathbf{v}_s\}$  of 0/1-vectors parameterised by $q$ as follows.
\begin{align*}
\mathbf{v}_r &= (\underbrace{0}_{\lceil q/2 \rceil },1,\underbrace{0}_{2q-\lceil q/2 \rceil})\\
\mathbf{v}_{r^{-1}} &= (\underbrace{0}_{\lceil q/2 \rceil -1},1,\underbrace{0}_{2q+1-\lceil q/2 \rceil})\\
\mathbf{v}_s &= (\underbrace{0}_{\lceil q/2 \rceil -1},1,1,\underbrace{0}_{2q - 1-2\lceil q/2 \rceil},1,1,\underbrace{0}_{\lceil q/2 \rceil -1}).
\end{align*}

For $x \in H$ and $y \in S$ the element of $H^k$ formed by replacing 1's in $\mathbf{v}_y$ by the element $x$ will be written as $\mathbf{v}_y(x)$.  For all $x \in H$, define elements of $G$ by
\begin{align*}
a(x) &= (\mathbf{v}_r(x);r)\\
A(x) &= (\mathbf{v}_{r^{-1}}(x);r^{-1})\\
b(x) &= (\mathbf{v}_s(x);s).  
\end{align*}

We claim that $X = \{ a(x),A(x),b(x):x \in H\}$ is a generating set for $G =  H^k \rtimes _{\varphi} D_{2k}$ such that every element of $G$ can be expressed as a product of exactly $k$ elements of $X$. Notice that $a(x)^{-1}=A(-x)$ and $b(x)^{-1}=b(-x)$, so that $X$ is identity-free and inverse-closed.  

A \emph{string} is a sequence $y_0,y_1,\ldots,y_{\ell-1}$ of length $\ell\leq k$ of elements of $S$.  The \emph{value} of the string is the element $y_0y_1 \ldots y_{l-1}$ of $D_{2k}$. With each string we may associate a general product $(\mathbf{v}_{y_0}(x_0);y_0)(\mathbf{v}_{y_1}(x_1);y_1)\cdots(\mathbf{v}_{y_{\ell-1}}(x_{\ell-1});y_{\ell-1})$, where each $x_i$ is an arbitrary element of $H$.  

Let $T = \{ t_0,t_1,\ldots,t_{\ell - 1}\} $ be a subset of $\Z_k$, representing a set of $\ell$ coordinate indices.  We will say that $T$ is \emph{free} for a string $y_0,y_1,\ldots ,y_{\ell-1}$ if for any $t$-tuple $(z_0,z_1,\ldots ,z_{t-1})$ of elements of $H$ there exist elements $x_0,x_1,\ldots ,x_{\ell - 1}$ of $H$ such that in the product $(\mathbf{v}_{y_0}(x_0);y_0)(\mathbf{v}_{y_1}(x_1);y_1)\cdots$ $(\mathbf{v}_{y_{\ell-1}}(x_{\ell-1});y_{\ell-1}) = (\mathbf{x'};y_0y_1\ldots y_{\ell - 1})$ the $t_i$-coordinate of $\mathbf{x'}$ is $z_i$ for $0 \leq i \leq \ell - 1$.  A string of length $k$ is \emph{good} if $\Z_k$ is free for the string.
We say that an element $r^is^j$ of $D_{2k}$ is \emph{covered} if there exists a good string with value $r^is^j$. 

To prove that every element of $G=H^k\rtimes_{\varphi}D_{2k}$ can be expressed as the product of exactly $k$ elements of $X$, it suffices to show that every element of $D_{2k}$ is covered.  Trivially the identity of $D_{2k}$ is covered by the string $r,r,\ldots ,r$.  Also, if a string $y_0,y_1,\ldots ,y_{k-1}$ is good, then so is the `inverse string' $y_{k-1}^{-1},\ldots ,y_1^{-1},y_0^{-1}$, so that if $r^i \in D_{2k}$ is covered, then so is $r^{-i}$.  Thus it is sufficient to show that the elements $r^i$, $1 \leq i \leq q$, and $r^is$, $0 \leq i \leq k-1$ are covered.  Since the proof requires consideration of 16 separate cases, we will work through one case in detail and display the good strings for the remaining cases in Tables \ref{tab:1mod4} and \ref{tab:3mod4}.  

Let $k\equiv 1\pmod{4}$ and put $q=2t$.  Let $i$ be an even integer such that $2 \leq i \leq q$ and consider the element $r^i$ of $D_{2k}$.  
For each such $i$ we have $r^i = s\underbrace{r}_{q+1-i}\underbrace{rsr^{-1}s}_{(i-2)/2}\underbrace{r}_{q+2-i}s$.  We claim that the string corresponding to this product is good.

Let $i \geq 4$.  The $t$-coordinate is free for $s$, so considering the first $s$ on the left, we see that the singleton
$\{t \} \underbrace{\psi _r}_{q+1-i}\underbrace{\psi _r\psi _s\psi _{r^{-1}}\psi _s}_{(i-2)/2}\underbrace{\psi _r}_{q+2-i}\psi _s 
= \{ t \} \underbrace{\psi _r}_{4t+1-i}\psi _s = \{ 4t- (5t+1-i)\} = \{ -t-1+i\}$ is free for the entire string.

The set $\{ t\}$ is free for $r$.  Therefore, considering the $j$-th $r$ in the first string $\underbrace{r}_{q+1-i}$, 
we see that $\{ t \} \underbrace{\psi _r}_{4t+1-i-j}\psi _s = \{ -t-1+i+j\}$ is free for the entire string 
for $1 \leq j \leq q+1-i$.

Expanding the product $(\mathbf{v}_r(x_1);r)(\mathbf{v}_s(x_2);s)(\mathbf{v}_{r^{-1}}(x_3);r^{-1})(\mathbf{v}_s(x_4);s)$, it is evident that the set $\{ t, t+1, 3t, 3t+1\} $ is free for the $j$-th substring $rsr^{-1}s$, $1 \leq j \leq (i-2)/2$.  It follows that the set  $\{ t, t+1, 3t, 3t+1\} \underbrace{\psi _r}_{2t-2j}\psi _s = \{ t+2j,t-1+2j,-t+2j,-t-1+2j\} $ is free for $1 \leq j \leq (i-2)/2$.

For $1 \leq j \leq q+2-i$, $\{ t \}$ is free for the $j$-th $r$ in the string $\underbrace{r}_{q+2-i}$ on the right.  
Hence for $1 \leq j \leq q+2-i$, $\{ t\} \underbrace{\psi _r}_{q+2-i-j}\psi _s = \{ t-2+i+j\}$ is free for the whole string. 

Finally, due to the presence of the final $s$, $\{ 3t+1\} $ is free for the whole string.

Let $0 \leq \ell_1 \leq \ell_2 < \ell_1' \leq \ell_2' \leq k-1$.  
If $T$ is a subset of $\Z_k$ which is free for the string $y_{\ell_1},y_{\ell_1+1},\ldots,y_{\ell_2}$ and $T'$ is a subset of $\Z_k$ which is free for the string 
$y_{\ell_1'},y_{\ell_1'+1},\ldots,y_{\ell_2'}$, where  the sets $T \psi_{\ell_2+1}\psi _{\ell_2+2}\cdots\psi _{k-1}$ 
and $T'\psi_{\ell_2'+1}\psi_{\ell_2'+2}\cdots\psi _{k-1}$ are disjoint, then their union:
\[(T \psi_{\ell_2+1}\psi_{\ell_2+2}\cdots\psi_{k-1}) \cup (T' \psi_{\ell_2'+1}\psi_{\ell_2'+2}\cdots\psi_{k-1})\]
 is free for any string of the form:
\[y_0,y_1,\ldots,y_{\ell_1},y_{\ell_1+1},\ldots,y_{\ell_2},y_{\ell_2+1},\ldots,y_{\ell_1'},y_{\ell_1'+1},\ldots,y_{\ell_2'},y_{\ell_2'+1},\ldots,y_{k-1}.\]

Collating the above free sets, it follows that the set $\{ -t+1,\ldots,3t+1\}= \Z_k$ is free for our string, so that the strings are good and the elements $r^i$ are covered for even $i$ in the range $4 \leq i \leq q$.  For $i = 2$, the string $sr^{2q-1}s$ is easily verified to be good.

In Table~\ref{tab:1mod4} we give the good strings for the remaining elements in the case $q\equiv 1\pmod{4}$.  Good strings for $k \equiv 3\pmod{4}$ and $q = 2t+1$ are displayed in Table~\ref{tab:3mod4}.

\begin{table}
\begin{tabular}{|p{0.95\textwidth}|}
\hline
{\bf Elements:} $r^i$, $i$ odd, $1 \leq i \leq q-1$    {\bf Good strings:} $\underbrace{rsr^{-1}s}_{(i-1)/2}\underbrace{r}_{q+1-i}s\underbrace{r}_{q-i}s$\\
{\bf Sketch of proof:} $\{ t,t+1,3t+1,3t+2\} $ is free for $rsr^{-1}s$, $\{ t\} $ is free for the first $s$ and $\{ 3t+1\} $ is free for the second $s$.\\
\hline
{\bf Elements:} $r^is$, $i$ even, $0 \leq i \leq q-2$    {\bf Good strings:} $\underbrace{rsr^{-1}s}_{i/2}\underbrace{r}_{q-i}s\underbrace{r}_{q-i}$\\
{\bf Sketch of proof:} $\{ t,t+1,3t+1,3t+2\} $ is free for $rsr^{-1}s$ and $\{ t\} $ is free for $s$.\\
\hline
{\bf Element:} $r^qs$      {\bf Good string:} $rr\underbrace{rsr^{-1}s}_{(q-2)/2}rsr$\\
{\bf Sketch of proof:} $\{ t,t+1,3t+1,3t+2\} $ is free for $rsr^{-1}s$ and $\{ 3t+1\} $ is free for $s$.\\
\hline
{\bf Elements:} $r^{-i}s$, $i$ odd, $1 \leq i \leq q-1$    {\bf Good strings:} $\underbrace{r}_{q+2-i}\underbrace{rsr^{-1}s}_{(i-1)/2}\underbrace{r}_{q-i}s$\\
{\bf Sketch of proof:} $\{ t,t+1,3t+1,3t+2\} $ is free for $rsr^{-1}s$ and $\{ 3t+1\} $ is free for $s$.\\
\hline
{\bf Elements:} $r^is$, $i$ odd, $1 \leq i \leq q-1$    {\bf Good strings:} $s\underbrace{r}_{q-i}\underbrace{rsr^{-1}s}_{(i-1)/2}\underbrace{r}_{q+2-i}$\\
{\bf Sketch of proof:} $\{ t,t+1,3t,3t+1\} $ is free for $rsr^{-1}s$ and $\{ t\} $ is free for $s$.\\
\hline
{\bf Element:} $r^{-q}s$      {\bf Good string:} $rsrr\underbrace{rsr^{-1}s}_{(q-4)/2}rsr^{-1}r^{-1}s$\\
{\bf Sketch of proof:} $\{ t,t+1,3t+1,3t+2\} $ is free for $rsr^{-1}s$, $\{ 3t+1\} $ is free for the first $s$ 
and $\{ t-1 \} $ is free for the second and third $s$.\\
\hline
{\bf Elements:} $r^{-i}s$, $i$ even, $2 \leq i \leq q-2$    {\bf Good strings:} $\underbrace{r}_{q-i}s\underbrace{r}_{q-1-i}\underbrace{rsr^{-1}s}_{i/2}r$\\
{\bf Sketch of proof:} $\{ t,t+1,3t,3t+1\} $ is free for $rsr^{-1}s$ and $\{ t\} $ is free for $s$.\\
\hline
\end{tabular}
\caption{Good strings for remaining elements of $D_{2k}$ in the case $k\equiv 1\pmod{4}$}
\label{tab:1mod4}
\end{table}

\begin{table}
\begin{tabular}{|p{0.95\textwidth}|}
\hline
{\bf Elements:} $r^i$, $i$ even, $2 \leq i \leq q-1$   {\bf Good strings:} $ s\underbrace{r}_{q+1-i}\underbrace{rsr^{-1}s}_{(i-2)/2}\underbrace{r}_{q+2-i}s$\\
{\bf Sketch of proof:} $\{ t+1,t+2,3t+2,3t+3\} $ is free for $rsr^{-1}s$, $\{ t+1\} $ is free for the first $s$ and $\{ 3t+2\} $ is free for the second $s$.\\
\hline
{\bf Elements:} $r^i$, $i$ odd, $1 \leq i \leq q-2$   {\bf Good strings:} $\underbrace{rsr^{-1}s}_{(i-1)/2}\underbrace{r}_{q+1-i}s\underbrace{r}_{q-i}s$\\
{\bf Sketch of proof:} $\{ t+1,t+2,3t+2,3t+3\} $ is free for $rsr^{-1}s$, $\{ 3t+2\} $ is free for $s$.\\
\hline
{\bf Element:} $r^q$       {\bf Good string:} $srsr\underbrace{rsr^{-1}s}_{(q-3)/2}rrr$\\
{\bf Sketch of proof:} $\{ t+1,t+2,3t+1,3t+2\} $ is free for $rsr^{-1}s$, $\{ t+1\} $ is free for $s$.\\
\hline
{\bf Elements:} $r^is$, $i$ even, $0 \leq i \leq q-1$  {\bf Good strings:} $\underbrace{rsr^{-1}s}_{i/2}\underbrace{r}_{q-i}s\underbrace{r}_{q-i}$\\
{\bf Sketch of proof:} $\{ t+1,t+2,3t+2,3t+3\} $ is free for $rsr^{-1}s$ and $\{ 3t+2\} $ is free for $s$.\\
\hline
{\bf Elements:} $r^{-i}s$, $i$ odd, $1 \leq i \leq q$  {\bf Good strings:} $\underbrace{r}_{q+1-i}\underbrace{rsr^{-1}s}_{(i-1)/2}\underbrace{r}_{q+1-i}s$\\
{\bf Sketch of proof:} $\{ t+1,t+2,3t+2,3t+3\} $ is free for $rsr^{-1}s$ and $\{ 3t+2\} $ is free for $s$.\\
\hline
{\bf Elements:} $r^is$, $i$ odd, $1 \leq i \leq q-2$  {\bf Good strings:} $s\underbrace{r}_{q-i}\underbrace{rsr^{-1}s}_{(i-1)/2}\underbrace{r}_{q+2-i}$\\
{\bf Sketch of proof:} $\{ t+1,t+2,3t+2,3t+3\} $ is free for $rsr^{-1}s$ and $\{ t+1\} $ is free for $s$.\\
\hline
{\bf Element:} $r^qs$       {\bf Good string:} $ sr\underbrace{rsr^{-1}s}_{(q-1)/2}r$\\
{\bf Sketch of proof:} $\{ t+1,t+2,3t+2,3t+3\} $ is free for $rsr^{-1}s$, $\{ 3t+1\} $ is free for $s$.\\
\hline
{\bf Elements:} $r^{-i}s$, $i$ even, $2 \leq i \leq q-1$ {\bf Good strings:} $\underbrace{r}_{q+1-i}s\underbrace{r}_{q+1-i}\underbrace{rsr^{-1}s}_{(i-2)/2}rr$\\
{\bf Sketch of proof:} $\{ t+1,t+2,3t+2,3t+3\} $ is free for $rsr^{-1}s$ and $\{ t+1\} $ is free for $s$.\\
\hline
\end{tabular}
\caption{Good strings for $k\equiv 3\pmod{4}$ and $q = 2t+1$}
\label{tab:3mod4}
\end{table}

Since all elements of $D_{2k}$ are covered,  for $m \geq 2$ we have Cayley graphs with odd diameters $k \geq 7$, 
degree $d = 3m$ and order  $n = 2km^k = (2k/3^k)d^k$.  
To extend to intermediate degrees, we may simply add one or two additional generators to our set $X$. 
Let $u = (0,0,0,0,\ldots,0;rs)$ and $v = (0,0,\ldots,0;r^2s)$ and $X_1 = X \cup \{ u\} $ and $X_2 = X \cup \{ u,v\}$. 
Then for $i = 1,2$, $\Cay(G,X_i)$ has diameter $k$, degree $d = 3m+i$ and order $n = (2k/3^k)(d-i)^k$.  
It follows that for all odd $k \geq 7$ and all degrees $d \geq 6$, there exists a Cayley graph of diameter $k$,
degree $d$ and order $(2k/3^k)d^k$.  

This completes the proof of the main result of this section, which we state here as a theorem.

\begin{theorem}\label{thm:general}
For odd $k \geq 7$, 
\[L^-(k)\geq\frac{2k}{3^k}.\]
\end{theorem}

\section{Concluding remarks}
We conclude with some brief remarks on our results and possible areas for future research. 
The diameter 3 results of Section~\ref{sec:matgroups} provide a useful improvement over the previous bound.
Matrix groups may be an interesting area to explore for other diameters, although there is no very obvious way to generalise our particular construction.

The techniques of Section~\ref{sec:semidirect} provide a framework which has the possibility for generating interesting constructions at a wider range of diameters.
However, as it stands the best results found could be considered ``sporadic'' in the sense that the particular groups and generating sets used for our
best constructions are the result of computer search. 
We note particularly that the puzzling non-monotonicity in the results of Section~\ref{sec:dir} seems likely to be caused by restrictions in the space
of possible construction parameters we were able to search.
Further progress might be made by considering a wider range of automorphisms of the direct product $H^k$ in our construction, 
although this would considerably increase the complexity of the search. Initial brief investigations in this direction have not as yet yielded any improved results.

A more interesting line of enquiry might be to extend the techniques of the general construction in Section~\ref{sec:general}, 
firstly to the case of even diameters, but also to try to reduce the denominator of $3^k$ in the asymptotic formula.
\bibliographystyle{plain}

\end{document}